\documentclass[12pt]{amsart}
\usepackage[colorlinks=true,citecolor=green,linkcolor=magenta]{hyperref}
\usepackage{amsmath}
\usepackage{verbatim}
\usepackage{amsfonts}
\usepackage{amssymb}
\usepackage{blkarray}
\usepackage{color,colortbl}
\usepackage[all]{xy}  
\usepackage{enumerate}
\usepackage[top=1in, bottom=1in, left=1in, right=1in]{geometry}
\usepackage{mathrsfs}

\usepackage{stmaryrd}
\usepackage{bbold}

\usepackage{cleveref}
\usepackage{soul}

\theoremstyle{definition}
\newtheorem{theorem}{Theorem}[section]
\newtheorem{theoremx}{Theorem}

\numberwithin{equation}{section}

\newtheorem{corollary}[theorem]{Corollary}
\newtheorem{lemma}[theorem]{Lemma}
\newtheorem{proposition}[theorem]{Proposition}

\theoremstyle{definition}
\newtheorem{definition}[theorem]{Definition}

\newtheorem{dis}[theorem]{Discussion}

\newtheorem{conjecture}[theorem]{Conjecture}
\newtheorem{remark}[theorem]{Remark}

\newtheoremstyle{TheoremNum}
        {8pt}{8pt}              
        {\upshape}                      
        {}                              
        {\bfseries}                     
        {.}                             
        {.5em}                             
        {\theoremname{#1}\theoremnote{ \bfseries #3}}
  \theoremstyle{TheoremNum}

\newcommand{\m}{\mathfrak{m}}

\newcommand{\Spec}{\operatorname{Spec}}

\newcommand{\Ass}{\operatorname{Ass}}

\newcommand{\id}{\operatorname{id}}

\renewcommand{\leq}{\leqslant}
\renewcommand{\geq}{\geqslant}

\title{Expected resurgence of ideals defining Gorenstein rings}

\author[Grifo]{Elo\'isa Grifo}
\address{Department of Mathematics, University of California, Riverside, Riverside, CA 92521, USA}
\email{eloisa.grifo@ucr.edu}

\author[Huneke]{Craig Huneke}
\address{Department of Mathematics, University of Virginia, Charlottesville, VA 22904-4135, USA}
\email{huneke@virginia.edu}

\author[Mukundan]{Vivek Mukundan}
\address{Department of Mathematics, Indian Institute of Technology Delhi, Delhi - 110016, India}
\email{vivekm85@gmail.com}


\begin{document}
\maketitle

\begin{abstract}
Building on previous work by the same authors, we show that certain ideals defining Gorenstein rings have expected resurgence, and thus satisfy the stable Harbourne Conjecture.
In prime characteristic, we can take any radical ideal defining a Gorenstein ring in a regular ring, provided its symbolic powers are given by saturations with the maximal ideal.
While this property is not suitable for reduction to characteristic $p$, we show that a similar result holds in equicharacteristic $0$ under the additional hypothesis that the symbolic Rees algebra of $I$ is noetherian.
\end{abstract}

\section{Introduction}\label{section intro}

In this paper we extend recent work by the same three authors \cite{GHM}, which studied what is called the stable Harbourne conjecture and its relationship to expected resurgence. Building on the sufficient conditions from \cite{GHM}, in the present paper we show that the stable Harbourne Conjecture holds for certain ideals defining Gorenstein rings. The original conjecture of Harbourne \cite{Seshadri,HaHu} concerns homogeneous ideals $I$ in $k[\mathbb{P}^n]$, $k$ a field, and their {\it symbolic powers} $I^{(n)} = \bigcap_{P\in\Ass(I)}(I^n R_P\cap R)$, and depends on the {\it big height} of $I$, the largest height (or codimension) of any minimal prime of $I$. We slightly rewrite the original conjecture here for radical ideals in any regular ring:

\begin{conjecture}[Harbourne]\label{harbourne} 
Let $I$ be a self-radical ideal of big height $c$ in a regular ring $R$.
Then for all $n \geqslant 1$,
$$I^{(cn-c+1)} \subseteq I^n.$$
\end{conjecture}

This conjecture is part of a larger program to find the best possible $a$ for each $b$ such that $I^{(a)} \subseteq I^b$, a program known as the Containment Problem. On the one hand, the containment $I^{(cn)} \subseteq I^n$ holds for all $n \geqslant 1$ by \cite{ELS}, \cite{comparison}, and \cite{MaSchwede}, but it is easy to find examples where these values are not best possible. However, Conjecture \ref{harbourne} can fail; Dumnicki, Szemberg and Tutaj-Gasi\'nska \cite{counterexamples} found a self-radical ideal $I$ of big height $2$ that is the first counterexample to $I^{(3)} \subseteq I^2$ for certain configurations of points in projective space $\mathbb{P}^2$. This example has been extended to entire classes of counterexamples coming from very special configurations in $\mathbb{P}^n$ \cite{HaSeFermat,ResurgenceKleinWiman,MalaraSzpond,RealsCounterexample,BenCounterexample,DrabkinSeceleanu}.

However, Harbourne's Conjecture is satisfied by various classes of ideals: those defining general points in $\mathbb{P}^2$ \cite{BoH} and $\mathbb{P}^3$ \cite{Dumnicki2015}, squarefree monomial ideals, or more generally ideals defining equicharacteristic rings with mild singularities \cite{GrifoHuneke}, such as Veronese or determinantal rings. More precisely, $I$ satisfies Harbourne's Conjecture whenever $R/I$ is F-pure in characteristic $p$, or of dense F-pure type in equicharacteristic $0$ \cite{GrifoHuneke}. There are even classes of ideals satisfying Harbourne's Conjecture over certain singular rings \cite{GrifoMaSchwede}.

There are, however, no known counterexamples to the following stable version of the conjecture:

\begin{conjecture}[Stable Harbourne]
Let $I$ be a self-radical ideal of big height $c$ in a regular ring $R$. Then for all $n \gg 0$,
$$I^{(cn-c+1)} \subseteq I^n.$$
\end{conjecture}

This stable version of Harbourne's Conjecture holds for very general and generic point configurations in $\mathbb{P}^n$ \cite[Theorem 2.2 and Remark 2.3]{StefanYuXie}, for sufficiently large sets of general points in $\mathbb{P}^n$ \cite[Corollary 4.3]{ChudnovskyGeneralPoints}, for ideals defining space monomial curves $(t^a, t^b, t^c)$ over fields of characteristic other than $3$ \cite[Theorem 4.3]{GHM}, and for homogeneous ideals generated in small degree in equicharacteristic $0$ \cite[Theorem 3.1]{GHM}. Moreover, various sufficient conditions for the stable Harbourne Conjecture are given in \cite{GrifoStable,GHM}. Many of these results rely on the {\it resurgence} of $I$, as defined by Bocci and Harbourne \cite{BoH}:

\begin{definition}[Bocci-Harbourne \cite{BoH}]
The resurgence of an ideal $I$ is given by 
$$\rho(I)=\sup \left\lbrace \frac{m}{s} ~|~I^{(m)}\not\subseteq I^s \right\rbrace.$$
\end{definition}

Notice that in particular, if $m > \rho(I)\cdot r$, then one is guaranteed that $I^{(m)}\subseteq I^r$. Related invariants have also been studied, such as the asymptotic resurgence \cite{AsymptoticResurgence}, which can be computed via integral closures \cite{AsymptoticResurgenceIntClosure}. The resurgence can often be bounded by other invariants \cite[Theorem 1.2.1]{BoH}, and sometimes even computed explicitly without completely solving the Containment Problem \cite{resurgence2,DHNSST2015,ResurgenceKleinWiman}. This is particularly interesting because, as noted in \cite[Remark 2.5]{GrifoStable}, the stable Harbourne Conjecture follows immediately as long as the resurgence of $I$ is strictly less than its big height. In this case, we say that $I$ has {\it expected resurgence}. In this paper, we expand on the sufficient conditions for expected resurgence from \cite{GHM} to prove the following main results:

\begin{theoremx}[Theorem \ref{charpGorenstein} and Theorem \ref{thm Gorenstein expected resurgence}]
Let $I$ be a homogeneous ideal of big height $c \geq 2$ in a standard graded regular ring $R$ with homogeneous maximal ideal $\m$, or in a regular local ring $(R, \m)$, and assume that $R$ contains a field $k$. Suppose that:
\begin{itemize}
\item $I^{(n)} = \left( I^n : \m^\infty \right)$ for all $n \geqslant 1$, and
\item $R/I$ is Gorenstein.
\end{itemize}
Then $I$ has expected resurgence whenever $k$ has prime characteristic or the symbolic Rees algebra of $I$ is noetherian.
\end{theoremx}

In the prime characteristic $p$ setting, the key step is to show that there exists $q = p^e$ such that $I^{(cq-c+1)} \subseteq \m I^{q}$; in fact, $I^{(cq-c+1)} \subseteq \m I^{[q]}$. Since this statement depends heavily on the characteristic, it is not suitable for reduction to characteristic $p$ techniques, which is why an additional assumption is needed in equicharacteristic $0$. However, the following strengthening of \cite{ELS,comparison} holds independently of the characteristic:

\begin{theoremx}[Theorem \ref{charpGorenstein} and Theorem \ref{theorem Gorenstein char 0}]
	Let $I$ be a homogeneous ideal of height $c \geq 2$ in a standard graded regular ring $R$ with homogeneous maximal ideal $\m$, or in a regular local ring $(R, \m)$, and assume that $R$ contains a field $k$. If $R/I$ is Gorenstein, then
	$$I^{(cn)} \subseteq \m I^n$$
	for all $n \geqslant 1$.
\end{theoremx}

While this might appear similar to \cite[Theorem 3.1]{TakagiYoshida}, there is a crucial difference between the two, since \cite[Theorem 3.1]{TakagiYoshida} only guarantees $I^{(ct-c+1)} \subseteq \m I^{t-1}$, and that one power difference is what will allow us to apply \cite[Theorem 3.3]{GHM} and conclude that $I$ satisfies expected resurgence. Our methods are also necessarily different from \cite[Theorem 3.1]{TakagiYoshida}, since that one extra power of $I$ is a priori difficult to obtain and directly uses the assumption that $R/I$ is Gorenstein.

In Section \ref{section prime char}, we establish the main result in the case of prime characteristic. In Section \ref{section expected resurgence}, we extend the methods of \cite{GHM} to a more general setting that will allow for our equicharacteristic $0$ results. Finally, we study the equicharacteristic $0$ case in Section \ref{section char 0}.


\section{Prime characteristic}\label{section prime char}

\begin{dis}\label{disc} Let $I$ be a homogeneous ideal of height $c \geq 2$ in a standard graded regular ring $R$ over a field $k$ of prime characteristic $p$ and with homogeneous maximal ideal $\m$,
or an ideal in a regular local ring $R$ of characteristic $p$ with maximal ideal $\m$. Assume that $R/I$ is Gorenstein.  We consider the minimal free resolution of $R/I$ over $R$:

$$\xymatrix{
0 \ar[r] & R \ar[r]^{\phi_c} & R^{\beta_{c-1}} \ar[r]^{\phi_{c-1}} & R^{\beta_{c-2}} \ar[r] & \cdots \ar[r] & R^{\beta_1} \ar[r]^{\phi_1} & R \ar[r] & R/I \ar[r] & 0.
}$$

Since $R$ is regular of prime characteristic $p$, \cite[Theorem 1.7]{P-S} shows that the complex of free modules

$$\xymatrix{
0 \ar[r] & R\ar[r]^{\phi^{[p]}_c}& R^{\beta_{c-1}}\ar[r]^{\phi^{[p]}_{c-1}} &R^{\beta_{c-2}}\ar[r]\ & \cdots \ar[r] & R^{\beta_1} \ar[r]^{\phi^{[p]}_1}  & R\ar[r] &R/I^{[p]}\ar[r] & 0
}$$
	
\noindent is also exact and gives a minimal free resolution of $R/I^{[p]}$.  After choosing bases of the free modules, the maps in this resolution are simply the entrywise $pth$ powers
of the entries of the corresponding matrices in the resolution of $R/I$.  The exactness of this resolution proves that $R/I^{[p]}$ is also Gorenstein.

The natural quotient map $R/I^{[p]} \longrightarrow R/I$ induces a map between the resolutions of $R/I$ and $R/I^{[p]}$, as follows:
\begin{align*}
	\xymatrix{
	0 \ar[r] & R \ar[r]^{\phi_c} & R^{\beta_{c-1}} \ar[r]^{\phi_{c-1}} & R^{\beta_{c-2}} \ar[r] & \cdots \ar[r] & R^{\beta_1} \ar[r]^{\phi_1} & R \ar[r] & R/I \ar[r] & 0	\\
	0 \ar[r] & R \ar[r]_-{\phi^{[p]}_c}\ar[u]^{\alpha_c} & R^{\beta_{c-1}}\ar[r]_-{\phi^{[p]}_{c-1}}\ar[u]_-{\alpha_{c-1}} & R^{\beta_{c-2}} \ar[r]\ar[u]_-{\alpha_{c-2}} & \cdots \ar[r] & R^{\beta_1} \ar[r]_-{\phi^{[p]}_1} \ar[u]^{\alpha_1} & R\ar[r] \ar@{=}[u] & R/I^{[p]} \ar[r] \ar[u] & 0	
}
\end{align*}
Notice that $\alpha_c=\mu_{\Delta}$, the map defined by multiplication by some element $\Delta$, which is unique up to homotopy. Since the entries of $\phi_c$ and $\phi_c^{[p]}$ for a minimal generating set for $I$ and $I^{[p]}$, respectively, it follows that $\Delta$ is unique modulo $I^{[p]}$. By standard linkage theory, $\Delta$ is a generator for the colon ideal $(I^{[p]}:I)$ modulo $I^{[p]}$. Duality shows that
$(I^{[p]}:\Delta ) = I$.
See \cite[Theorem 1.2]{KustinMillerLinkage} for details.
\end{dis}

A critical observation that we need for our main results is the following proposition:

\begin{proposition}\label{critical-prop} 
Let the notation be as in Discussion \ref{disc}.  Then $\Delta\cdot I\subseteq \m\cdot I^{[p]}$.
\end{proposition}

\begin{proof}  From the commutative diagram in Discussion \ref{disc}, we see that $\phi_c\circ \alpha_c=\alpha_{c-1}\circ \phi_c^{[p]}$. Thus $I(\Delta)\subseteq I^{[p]} \cdot I_1(\alpha_{c-1})$. Thus it is enough to show $I_1(\alpha_{c-1})\subseteq \m$. Suppose $I_1(\alpha_{c-1})=R$. Possibly after a suitable change of basis, we may assume that $\alpha_{c-1}=\id\oplus\alpha'$, and if $e_1$ denotes the first standard basis element of $R^{\beta_{c-1}}$, $\alpha_{c-1}(e_1) = e_1$. Then $\phi_{c-1}\circ\alpha_{c-1}(e_1)=\phi_{c-1}(e_1)$. If $\phi_{c-1}(e_1) = (v_1, \ldots, v_{\beta{c-2}})$, notice that all $v_i \in \m$, since $I_1(\phi_{c-1}) \subseteq \m$ (note that $c\geq2$). On the other hand, from the commutative diagram above we get
$$\phi_{c-1}\circ \alpha_{c-1}=\alpha_{c-2}\circ\phi_{c-1}^{[p]} \textrm{ so } \phi_{c-1}(e_1)=\alpha_{c-2} \circ \phi_{c-1}^{[p]}(e_1).$$
But then $v_i \in (v_1, \ldots, v_{\beta_{c-2}})^{[q]}$, which is a contradiction.
\end{proof}

\begin{theorem}\label{charpGorenstein} Let $I$ be a homogeneous ideal of height $c \geq 2$ in a standard graded regular ring $R$ containing a field $k$ of characteristic $p$ and with homogeneous maximal ideal $\m$,
or an ideal in a regular local ring $R$ of characteristic $p$ with maximal ideal $\m$. Assume that $R/I$ is Gorenstein. Then $I^{(cq-c+1)}\subseteq\m I^{[q]}$ for every  $q=p^e > pc$. Moreover, for every $n\geq 1$, $I^{(cn)} \subseteq \m I^n$.
\end{theorem}
\begin{proof} We show the first statement. Equivalently, we will show that $I^{(cqp-c+1)}\subseteq \m I^{[qp]}$ for all  $q = p^e > c$.  In this case, 
	$$I^{(cqp-c+1)}\subseteq I^{(cqp-q+1)},$$
	and by \cite[Theorem 3.1]{TakagiYoshida} using $n=1, k=cqp-q-c$, we have 
	$$I^{(cqp-q+1)}\subseteq \m I^{(cqp-q-c+1)}.$$ 
	On the other hand, 
	$$I^{(cqp-q-c+1)}I^{[q]}\subseteq I^{(cqp-c+1)}\subseteq I^{[qp]}.$$  
	The second containment is well-known; see for example \cite[Lemma 2.6]{GrifoHuneke}. Thus we have
	\begin{align*}
	I^{(cqp-c+1)}\subseteq \m I^{(cqp-q-c+1)}\subseteq \m (I^{[qp]}:I^{[q]})= \m (I^{[p]}:I)^{[q]}.
	\end{align*}
	
	Since $R/I$ is Gorenstein, Discussion \ref{disc} shows that $\left( I^{[p]} : I \right) = I^{[p]}+(\Delta)$ for some $\Delta\in R$. Since $I^{(cqp-c+1)} \subseteq I^{[qp]}$, we also have
	\begin{align*}
	I^{(cqp-c+1)}\subseteq \m (I^{[p]}:I)^{[q]}\cap I^{[qp]}\subseteq \m(I^{[qp]}+(\Delta^q))\cap I^{[qp]}\subseteq \m I^{[qp]}+\m(\Delta^q)\cap I^{[qp]}.
	\end{align*}
	Thus it is enough to show that 
\begin{align}\label{eq1maintheorem in char p}
	(\Delta^q)\cap I^{[qp]}\subseteq \m I^{[qp]}.
\end{align}
	If $r\Delta^q\in  I^{[qp]}$, then 
	\begin{align*}
	r \in I^{[qp]}: \Delta^q  = (I^{[p]}:\Delta)^{[q]} = I^{[q]}.
	\end{align*}

The last equality follows as in Discussion \ref{disc}.
	This shows that $(\Delta^{q})\cap I^{[qp]}\subseteq (I \Delta)^{[q]}$. By Proposition \ref{critical-prop},  $I(\Delta)\subseteq \m I^{[p]}$. The first statement follows.

To prove the second statement, assume not, and choose an element $f\in I^{(cn)}$, with $f\notin \m I^n$.
By the first part of the theorem, there exists $q > n$ such that $I^{(cq-c+1)} \subseteq \m I^{[q]}$. Observe that
$$I^{(cnq)} = I^{(cqn)} = I^{(cn + (cq-c)n))} \subseteq \left( I^{(cq-c+1)} \right)^n,$$
where the last containment follows by \cite[Theorem 1.1 (a)]{comparison}. Since $I^{(cq-c+1)} \subseteq \m I^{[q]}$,
$$I^{(cnq)} \subseteq \left( I^{(cq-c+1)} \right)^n \subseteq \left( \m I^{[q]} \right)^n = \m^n \left( I^{[q]} \right)^n.$$

Take any $z \in \left( \m^{[q]} : \m^n \right).$ Then
$$f^q \in \left( I^{(cn)} \right)^{q} \subseteq I^{(cnq)} \subseteq \m^n \left( I^{[q]} \right)^n,$$
so
$$z f^q \in \left( \m I^n \right)^{[q]}.$$
Then
$$z \in \left( \left( \m I^n \right)^{[q]} : f^q \right) = \left( \m I^n : f \right)^{[q]} \subseteq \m^{[q]},$$
since we assumed that $f \notin \m I^n$. But this implies that
$$\left( \m^{[q]} : \m^n \right) = \m^{[q]},$$
which is impossible by choice $q > n$. 
\end{proof}

\begin{remark}
	In the proof above, we used \cite[Theorem 3.1]{TakagiYoshida}, which says that if $(R, \m)$ is an excellent regular local ring of characteristic $p$ and $I$ is a radical ideal of big height $I$ in $R$, then
	$$I^{((h+k)n + 1)} \subseteq \m \left( I^{(k+1)} \right)^n$$
	for all $n \geqslant 1$ and $k \geqslant 0$. We note, however, that this result also holds in the case where $R$ is instead a standard graded regular ring containing a field $k$ of characteristic $p$ and with homogeneous maximal ideal $\m$. Indeed, the graded case follows directly from the local case, once we notice that it is sufficient to check that 
	$$\left( I^{((h+k)n + 1)} \right)_P \subseteq \left( \m \left( I^{(k+1)} \right)^n \right)_P$$
	at each prime ideal $P$ of $R$; when $P = \m$, the statement is precisely \cite[Theorem 3.1]{TakagiYoshida}, while for $P \neq \m$ the statement becomes 
	$$I^{((h+k)n + 1)} \subseteq \left( I^{(k+1)} \right)^n,$$
	which is \cite[Theorem 1.1 (a)]{comparison}.
\end{remark}

We recall \cite[Theorem 3.3]{GHM}:

\begin{theorem}[Theorem 3.3 in \cite{GHM}]\label{main resurgence}
Let $I$ be a radical ideal of big height $c \geqslant 2$ in a regular local ring $(R,\m)$ containing a field, or a quasi-homogeneous radical ideal of big height $c\geqslant 2$ in a polynomial ring over a field with
	irrelevant maximal ideal $\m$.
	 If 
	$$
	I^{(ct-c+1)} \subseteq \m I^t \text{ for some fixed }t,
	$$
	and if $I_p$ has the property that $I_P^{(n)} = I_P^n$ for all $P \ne \m$ and for all $n$, then
	$\rho(I) < c$.
\end{theorem}

As a consequence, we can now prove the following:

\begin{corollary}
Let $I$ be a radical  ideal of height $c \geq 2$ in a regular local ring $(R,\m)$ of characteristic $p$. Suppose that $R/I$ is Gorenstein. Further assume that for all primes $P$ not equal to
the maximal ideal $\m$, $I^n_P = I^{(n)}_P$. Then the resurgence of $I$ is expected, i.e., $\rho(I) < c$.  In particular, $I$ satisfies the stable Harbourne conjecture.
\end{corollary}

\begin{proof}  
Using Theorem \ref{charpGorenstein}, we obtain that $I^{(cq-c+1)}\subseteq\m I^{[q]}$ for every  $q=p^e > pc$.  Then simply put $t= q$ for large enough $q$ in the statement of Theorem \ref{main resurgence}.
\end{proof}

\begin{remark}  
The condition that $I^n_P = I^{(n)}_P$ for all primes $P$ not equal to the maximal ideal comes up often in the statements of our theorems. It is not an easy question in general to decide whether or not the symbolic powers of an ideal agree with the usual powers, although this equality holds whenever $I$ is generated by a regular sequence. In particular, if $I$ is self-radical in a regular ring $R$ and the dimension of $R/I$ is one, then $I_P$ is generated by a regular sequence for all primes $P$ not equal to the maximal ideal. More generally if $R/I$ has an isolated singularity, or in the graded case if the associated projective variety is smooth, then $I^n_P = I^{(n)}_P$ for all $n$ and for all primes $P$ with $P \ne \m$. This condition is also equivalent to $I^{(n)} = (I^n : \m^\infty)$ for all $n \geqslant 1$.
\end{remark}

\begin{remark}
	One of the important evidences in the proof of \Cref{charpGorenstein} is the containment \eqref{eq1maintheorem in char p}. When $R$ is graded, then the containment \eqref{eq1maintheorem in char p} can also be verified by comparing the degree of the generators on both sides of the relation. Consequently, it would be enough to check if $\deg \Delta$ is more than the maximal degree of a minimal generating set of the homogeneous ideal $I$. When translated into the language of Betti numbers, the previous statement is true due the existence of \textit{extremal Betti numbers} of $I$. Roughly speaking, an extremal Betti number $\beta_{i,j}$ is the non-zero top left corner in a block of zeroes in the betti diagram for $I$ (we refer to \cite{BaChaPop} for a formal definition). The existence of extremal Betti numbers has been proved in \cite[Theorem 1.2]{BaChaPop}.
\end{remark}


\section{Expected Resurgence}\label{section expected resurgence}

In this section, we extend some the sufficient conditions for expected resurgence from \cite{GHM} in such a way that will allow us to prove expected resurgence for suitable ideals defining Gorenstein rings in equicharacteristic $0$. We also search for other characteristic-free conditions which imply expected resurgence, focusing on the condition that the symbolic Rees algebra of $I$ is noetherian. Considering such a condition is particularly important since even a strong result in characteristic $p$ will not necessarily give any conclusion in equicharacteristic $0$, despite of the usual technique of reduction to characteristic $p$. This is due to the fact that resurgence is an asymptotic invariant, and one can only generally conclude something for all large values. In characteristic $p$, these large values often depend on the characteristic itself, for example as in the statement of Theorem \ref{charpGorenstein}. As the characteristic grows, so too do these bounds, making it impossible to apply standard reduction techniques. 

The {\it symbolic Rees algebra} of an ideal $I$ in a ring $R$ is the graded algebra
$$\mathcal{R}_s(I) = \bigoplus_{n \geqslant 0} I^{(n)} t^n \subseteq R[t].$$
In general, the symbolic Rees algebra can fail to be noetherian, even if $I$ is a prime ideal of height $d-1$ in a regular ring of dimension $d$, with the first such example with non-noetherian symbolic Rees algebra found by Roberts in \cite{RobertsExample}.  Even for the special case of the defining ideal of a space monomial curve $(t^a, t^b, t^c)$, the symbolic Rees algebras are not noetherian \cite{NonnoetherianSymb} for certain choices of $a$, $b$, and $c$. The condition that the symbolic Rees algebra is noetherian does have nice consequences, such as the fact that the resurgence is rational \cite{DiPasqualeDrabkin}.

While determining whether or not the symbolic Rees algebra of a given radical ideal is noetherian is a difficult question in general, this does indeed happen for various interesting classes of ideals. For example, the symbolic Rees algebra of a monomial ideal is always noetherian \cite[Proposition 1]{LyuArithMon}. Morales gave a criterion for when primes defining curves have noetherian symbolic Rees algebras \cite{Morales}; see also \cite{HunekeHilbertSymb}. In \cite{Cutkosky}, Cutkosky showed that the symbolic Rees algebra of the defining ideal of $k[t^a,t^b,t^c]$ is noetherian as long as $(a+b+c)^2 > abc$.

\vspace{2em}

We first need a result similar to \cite[Lemma 2.5]{GHM}, but stronger:

\begin{lemma}\label{lemma cn in mI implies expected resurgence}
Let $(R, \m)$ be a regular local ring of dimension $d$, or $R$ a standard graded ring containing a field $k$ with homogeneous maximal ideal $\m$. Let $I$ be an ideal of big height $c$ such that $I^{(n)} = \left( I^n : \m^\infty \right)$ for all $n \geqslant 1$. If $I^{(cn)} \subseteq \overline{m^{\lfloor \frac{n}{\alpha} \rfloor} I^n}$ for all $n \gg 0$, then $\rho(I)<c$.
\end{lemma}

\begin{proof}
By \cite[Lemma 2.5]{GHM}, it is enough to find positive $t < r$ such that $I^{(tcn)} \subseteq I^{rn}$ for all $n\gg 0$. Since $\overline{I^{rn+d}} \subseteq I^{rn}$, it is enough to show that $I^{(ctn)} \subseteq \overline{I^{rn+d}}$ for all $n\gg 0$. 
	
	Now consider any positive integer $t$, and let $v \in R$ be nonzero such that
	$$v \overline{\left( \m^{\lfloor\frac{t}{\alpha}\rfloor} I^{t} \right)^n} \subseteq \m^{\lfloor\frac{t}{\alpha}\rfloor n} I^{nt}$$
	for all $n$. For why such a $v$ exists, see for example \cite[Corollary 6.8.12]{HunekeSwansonIntegral2006}.	
	Fix integers $l$ and $k$ with the following properties:
	\begin{enumerate}
	\item $\m^{ln} I^{(n)} \subseteq I^n$ for all $n \geqslant 1$.
	\item If $vJ \subseteq \overline{I^n}$ for some ideal $J$, then $J \subseteq \overline{I^{n-k}}$.
	\end{enumerate}
	Such $l$ exists by \cite[Main Theorem]{Swanson97}, given our assumption that $\m$ is the only possible associated prime for $I^n$.  The integer $k$ exists since the integral closures of powers of $I$ are
	determined by the (finitely many) Rees valuations $v_1,...,v_m$ of $I$; see \cite[Theorem 10.2.2]{HunekeSwansonIntegral2006}. If $k$ is chosen so that $v_i(v) \leq v_i(I^k)$ for every such $v_i$, then $vJ \subseteq \overline{I^n}$ implies that
	$v_i(J)\geq v_i(I^{n-k})$ for $1\leq i\leq m$, which then implies that $J \subseteq \overline{I^{n-k}}$.

	Given any integer $b$, and $n \gg 0$,
	\begin{align}\label{thm1eq2}
	(v I^{(ctn)})^{b+1} & \subseteq \left(v \, \overline{\m^{\lfloor\frac{t}{\alpha}\rfloor n} I^{tn}}\right)^b I^{(ctn)} \subseteq \m^{b \lfloor\frac{t}{\alpha}\rfloor n} I^{bnt} I^{(ctn)}.
	\end{align}
	Now we claim that we can chose $b$ (independent of $t$ and $r$) such that
	$$b\left\lfloor\frac{t}{\alpha}\right\rfloor > lct.$$
	To do that, take $b \geqslant 2  l c \alpha$.

	\begin{align*}
	b\left\lfloor\frac{t}{\alpha}\right\rfloor=bu&>2lc\alpha u\\
	&=lc\alpha u+lc\alpha u\\
	&=lc(t-v)+lc\alpha u\\
	&=lct-lcv+lc\alpha u=lct+lc(\alpha u-v)\\
	&>lctn.
	\end{align*}
	
	With our choice of $b$, we have
	$$\m^{b \lfloor\frac{t}{\alpha}\rfloor n} I^{(cnt)} \subseteq \m^{lctn} I^{(cnt)} \subseteq I^{ctn},$$
	and thus
	$$(v I^{(ctn)})^{b+1} \subseteq I^{ctn + btn}.$$
	So far, everything we have shown holds for any positive integer $t$. We will now show that given a fixed $b \geqslant 2 l c \alpha$ and $k$ as above, we can find positive integers $t < r$ such that
	$$ctn+btn > (rn + d + k)(b+1).$$
	With such values, we will have
	$$(v I^{(ctn)})^{b+1} \subseteq \left( I^{rn+d+k} \right)^{b+1},$$
	which shows that
	$$v I^{(ctn)} \subseteq \overline{I^{rn+d+k}}.$$
	As a consequence, we have
	$$I^{(ctn)} \subseteq \overline{I^{rn+d}} \subseteq I^{rn}.$$
	So all that remains to show that given $b \geqslant 2 l c \alpha$, we can chose positive $t<r$ such that
	$$ctn+btn > (rn + d + k)(b+1).$$
	To do that, choose an integer $r \geqslant 3$ such that $r \geq \frac{1+\gamma (d+l)}{1-\gamma}$, where $\gamma=\frac{b+1}{b+c}$. Then
\begin{align*}
r- r \gamma \geq 1+\gamma (d+k) & \textrm{ so } \\
r-1 \geqslant (r + d + k) \gamma 
\end{align*}
Pick any integer $t$ such that 
\begin{align*}
r > t \geq (r + d + k) \gamma.
\end{align*}
For all $n \geq 1$, we have
\begin{align}
t & \geq \left(r + d + k \right) \gamma \geqslant \left(r+\frac{d+k}{n}\right)\gamma = \left(r + \frac{d+k}{n} \right) \left(\frac{b+1}{b+c}\right).
\end{align}
Multiplying both sides by $n$, we get
\begin{align*}
tn & \geqslant \left(rn + d + k \right) \left(\frac{b+1}{b+c}\right),\textrm{ so } \\
tn(b+c) & \geqslant (rn+d+k)(b+1).
\end{align*}

And this concludes our proof that
$$I^{(ctn)} \subseteq I^{rn}.$$
By \cite[Lemma 2.5]{GHM}, this implies that $\rho(I) < c$.
\end{proof}

Our main result in this section assumes the symbolic Rees algebra of $I$ is noetherian. Under mild circumstances such as excellence, this
is equivalent to the statement that there exists $l$ such that $I^{(ln)} = \left( I^{(l)} \right)^n$ for all $n \geqslant 1$. See for example \cite[Lemma 2]{Rees1958ProblemZariski} or \cite[Lemma 5.2]{mythesis}. 

\begin{theorem}
	Let $I$ be a homogeneous ideal of big height $c \geq 2$ in a standard graded regular ring $R$ with homogeneous maximal ideal $\m$, or an ideal in a regular local ring $(R, \m)$. Suppose that:
	\begin{enumerate}[a)]
		\item $I^{(cn)} \subseteq \overline{\m I^n}$ for all $n \gg 0$; and 
		\item the symbolic Rees algebra of $I$ is noetherian. 
	\end{enumerate}
Then there exists an integer $l$ such that $I^{(cn)}\subseteq \overline{\m^{\left\lfloor \frac{n}{l} \right\rfloor} I^{n}}$ for all $n\geq 1$.
\end{theorem}

\begin{proof}
Assumption b) implies that there exists $t$ such that $I^{(tn)} = \left( I^{(t)} \right)^n$ for all $n \geqslant 1$; see for example \cite[Lemma 2]{Rees1958ProblemZariski}. Notice that this containment also holds if we replace $t$ by any multiple of $t$, and so we may assume without loss of generality that $l=tk$ is chosen large enough so that 
\begin{itemize}
	\item $I^{(cn)} \subseteq \overline{\m I^n}$ for all $n \geqslant l$, and
	\item $I^{(ln)} = \left( I^{(l)} \right)^n$ for all $n \geqslant 1$.
\end{itemize}
First, we will show our claim when $n$ is a multiple of $l$, say $n = lk$. Indeed,
$$I^{(clk)} = \left( I^{(cl)} \right)^k \subseteq \left( \overline{\m I^{l}} \right)^k \subseteq \overline{\m^k I^{lk}} \subseteq \overline{\m^{\left\lfloor \frac{kl}{l} \right\rfloor} I^{kl}}.$$
Now take $f \in I^{(cn)}$ and $n \geqslant 1$ . Then
$$f^l \in \left( I^{(cn)} \right)^l \subseteq I^{(cnl)} \subseteq \overline{\m^{n} I^{ln}} \subseteq \overline{\left( \m^{\left\lfloor \frac{n}{l} \right\rfloor} I^{n} \right)^l}.$$
Thus $f \in \overline{\m^{\left\lfloor \frac{n}{l} \right\rfloor} I^{n}}$. We have now shown that
$$I^{(cn)} \subseteq \overline{\m^{\left\lfloor \frac{n}{l} \right\rfloor} I^{n}}$$
for all $n \geq 1$,	
\end{proof}

In fact, the conclusion of the above theorem is stronger than what we need to apply \Cref{lemma cn in mI implies expected resurgence} and conclude that $I$ has expected resurgence.

\begin{corollary}\label{lemma noetherian symbolic rees algebra implies expected resurgence}
Let $I$ be a homogeneous ideal of big height $c \geq 2$ in a standard graded regular ring $R$ with homogeneous maximal ideal $\m$, or in a regular local ring $(R, \m)$. Suppose that:
\begin{enumerate}[a)]
\item $I^{(n)} = \left( I^n : \m^\infty \right)$ for all $n \geqslant 1$;
\item $I^{(cn)} \subseteq \overline{\m I^n}$ for all $n \gg 0$; and 
\item the symbolic Rees algebra of $I$ is noetherian. 
\end{enumerate}
Then the resurgence of $I$ is expected.
\end{corollary}

\begin{proof}
The result follows from the previous theorem and Lemma \ref{lemma cn in mI implies expected resurgence}.
\end{proof}


\section{Resurgence in Characteristic Zero for Ideals Defining Gorenstein Rings}\label{section char 0}

We now apply the results from Section \ref{section expected resurgence} to ideals defining Gorenstein rings in equicharacteristic $0$.

\begin{theorem}\label{theorem Gorenstein char 0}
Let $I$ be a homogeneous ideal of height $c \geq 2$ in a standard graded regular ring $R$ containing a field $k$ of characteristic $0$ and with homogeneous maximal ideal $\m$, or an ideal in a regular local ring $R$ of characteristic $p$ with maximal ideal $\m$. If $R/I$ is Gorenstein, then for all $n\geq 1$ we have
$$I^{(cn)} \subseteq \m I^n.$$
\end{theorem}

\begin{proof}
 Suppose that the theorem is false, and pick $n$ and $f$ such that $f \in I^{(cn)}$ but $f \notin \m I^n$.  First assume we are in either the graded case, or the case in which $R$ is essentially of finite type over
$k$. We will reach a contradiction  via standard reduction to characteristic $p$ techniques, similarly to \cite[Theorem 4.4]{comparison}. Using the standard descent theory of \cite[Chapter 2]{HHCharZero}, we replace the field $k$ by a finitely generated $\mathbb{Z}$-algebra $A$, so that we have a counterexample for this containment in an affine $A$-algebra $R_A$ over $A$ with $R_A \subseteq R$ and $R \cong k \otimes_A R_A$; after localizing at a nonzero element of $A$, we may assume that $A$ is smooth over $\mathbb{Z}$. We also retain models $I_A$ of $I$ and $\m_A$ of $\m$, such that $R_A/I_A$ is Gorenstein and such that there is an
element $f_A\in I_A^{(cn)}$ such that $f_A\notin \m_AI_A^n$. The result now follows from the fact that for almost all fibers, the containment holds for the map $\mathbb{Z} \longrightarrow R_A$ after passing to fibers over closed points of $\Spec (\mathbb{Z})$. Suppose that isn't so. Further using \cite[Theorem 2.3.15]{HHCharZero} we now have a regular ring $R_p$ containing a field $k_p$ of characteristic $p$ and with (homogeneous) maximal ideal $\m_p$, an ideal $I_p$ of height $c$ such that $R_p/I_p$ is Gorenstein, and $f_p \in I_p^{(cn)}$ such that $f_p \notin \m_p I_p^n$.
This contradicts the conclusion of Theorem \ref{charpGorenstein}.
This shows that for our original $I$, in equicharacteristic $0$, and for all $n \geqslant 1$, $I^{(cn)} \subseteq \m I^n$. 

In the general case we also follow the detailed proof in \cite{comparison}.  The point is that after completing and using Artin-Approximation, we can descend a counterexample to the desired containment.  The only
difference in this case from a similar result in \cite{comparison} is the added condition that $R/I$ be Gorenstein. However, as described in Chapter 2 of \cite{HHCharZero}, we can preserve the shape of resolutions as
we descend, and in particular can preserve the Betti numbers in a resolution of  $R/I$ over $R$.  Since the height does not change (even though the dimension of ambient regular ring may change), we can preserve
the fact that $R/I$ is Gorenstein. This reduces to the case in which $R$ is essentially of finite type over $k$, and the first part of the proof does this case.
\end{proof}

\begin{corollary} Let $I$ be a homogeneous ideal of height $c \geq 2$ in a standard graded regular ring $R$ containing a field $k$ of characteristic $0$ and with homogeneous maximal ideal $\m$, or an ideal in a regular local ring $R$ of characteristic $p$ with maximal ideal $\m$. Assume:
	\begin{enumerate}[a)]
		\item $I^{(n)}=(I^n:\m^\infty)$ for all $n\geq 1$,
		\item the symbolic Rees algebra of $I$ is noetherian and
		\item $R/I$ is Gorenstein.
	\end{enumerate}
Then the resurgence of $I$ is expected.
\end{corollary}

\begin{proof} Recall that our assumption that the symbolic Rees algebra is noetherian implies that there exists $l$ such that $I^{(ln)} = \left( I^{(l)} \right)^n$ for all $n \geqslant 1$. By Theorem \ref{theorem Gorenstein char 0}, our conditions imply that for all $n \geqslant 1$, $I^{(cn)} \subseteq \m I^n$. Then by \Cref{lemma noetherian symbolic rees algebra implies expected resurgence}  $I$ has expected resurgence.  \end{proof}

As we observed earlier, if $I$ has expected resurgence, then $I$ satisfies the stable Harbourne conjecture, which is an apparently much weaker condition.  However, in characteristic $0$, the assumption that the symbolic Rees algebra is noetherian gives the converse, so that the conditions of expected resurgence and stable Harbourne are in fact equivalent:

\begin{theorem}\label{thm Gorenstein expected resurgence}
Let $I$ be a homogeneous ideal of big height $c \geq 2$ in a standard graded regular ring $R$ with homogeneous maximal ideal $\m$, or an ideal in a power series ring over a field of characteristic $0$ with
maximal ideal $\m$. Suppose that:
\begin{enumerate}[a)]
\item $I^{(n)} = \left( I^n : \m^\infty \right)$ for all $n \geqslant 1$, and
\item the symbolic Rees algebra of $I$ is noetherian.
\end{enumerate}
If $I$ satisfies the stable Harbourne Conjecture, then the resurgence of $I$ is expected.
\end{theorem}

\begin{proof}
We claim that for all $n \gg 0$, 
$$I^{(cn)} \subseteq \overline{\m I^n}.$$
First, note that in characteristic $0$, any element $f$ is integral over $\m D(f)$, where $D(f)$ is the ideal generated by the partial derivatives of $f$. In particular, $I^{(n)} \subseteq \overline{\m I^{(n-1)}}$ for all $n$.
Suppose that $n$ is large enough for stable Harbourne to hold, meaning $I^{(cn-c+1)} \subseteq I^n$. Then
\begin{align*}
I^{(cn)} & \subseteq \overline{\m I^{(cn-1)}} & \textrm{as we just showed} \\
& \subseteq \overline{\m I^{(cn-c+1)}} & \textrm{because } c \geqslant 2 \\
& \subseteq \overline{\m I^n} & \textrm{by assumption}
\end{align*}
This shows we can apply \Cref{lemma noetherian symbolic rees algebra implies expected resurgence}, and $I$ must have expected resurgence.
\end{proof}

\begin{remark}
In fact, the proof above shows that under these assumptions, expected resurgence is equivalent to $I^{(cn-1)} \subseteq I^n$ for $n \gg 0$. Essentially, this says that $I$ has expected resurgence if and only if $I^{(cn)} \subseteq I^n$ is not the best possible containment as $n \longrightarrow \infty$.
\end{remark}

\section*{Acknowledgements}
This work was done while the first author was on a year long visit to the University of Utah; she warmly thanks the University of Utah Math Department for their hospitality. The first author is supported by NSF grant DMS-2001445.

\bibliographystyle{alpha}
\bibliography{References}
\end{document}